\documentclass[10pt,a4paper,reqno]{amsart}

\usepackage{amssymb}
\usepackage{enumitem}
\usepackage[ruled,noline,titlenumbered,linesnumbered,noend,norelsize]{algorithm2e}
\DontPrintSemicolon
\SetNlSty{footnotesize}{}{}
\IncMargin{0.5em}
\SetNlSkip{1em}
\SetKwIF{If}{ElseIf}{Else}{if}{then}{else if}{else}{endif}
\SetKw{Return}{return}
\SetKw{Pick}{pick}
\SetNlSkip{2em}
\newenvironment{algorithm-hbox}{\hbadness=10000\begin{algorithm}}{\end{algorithm}}

\theoremstyle{plain}
\newtheorem{theorem}{Theorem}

\newtheorem{lemma}[theorem]{Lemma}
\newtheorem{corollary}[theorem]{Corollary}

\newtheorem{prop}[theorem]{Proposition}

\numberwithin{equation}{section}

\renewcommand{\epsilon}{\varepsilon}

\renewcommand{\leq}{\leqslant}
\renewcommand{\geq}{\geqslant}

\newcommand{\Nat}{\mathbb{N}}

\newcommand{\evS}{\mathcal{A}}
\newcommand{\vS}{\mathcal{X}}
\newcommand{\vars}{\mathrm{vbl}}
\newcommand{\prob}{\mathrm{Pr}}

\begin{document}
\title{Improved algorithms for colorings of simple hypergraphs and applications}
\author{Jakub Kozik, Dmitry Shabanov}
\address{Theoretical Computer Science Department, Faculty of Mathematics and Computer Science, Jagiellonian University, Krak\'{o}w, Poland}
\email{Jakub.Kozik@uj.edu.pl}
\address{Department of probability Theory, Faculty of Mechanics and Mathematics, Lomonosov Moscow State University, Moscow, Russia}
\email{dm.shabanov.msu@gmail.com}

\thanks{Research of J.\ Kozik  was supported by Polish National
Science Center within grant 2011/01/D/ST1/04412.}

\begin{abstract}
The paper deals with extremal problems concerning colorings of hypergraphs.
By using a random recoloring algorithm we show that any $n$-uniform simple (i.e. every two distinct edges share at most one vertex) hypergraph $H$ with maximum edge degree at most
\[
  \Delta(H)\leq c\cdot nr^{n-1},
\]
is $r$-colorable, where $c>0$ is an absolute constant.

As an application of our proof technique we establish a new lower bound for Van der Waerden number $W(n,r)$, the minimum $N$ such that in any $r$-coloring of the set $\{1,\ldots,N\}$ there exists a monochromatic arithmetic progression of length $n$.
We show that
\[
  W(n,r)>c\cdot r^{n-1},
\]
for some absolute constant $c>0$.

\end{abstract}

\maketitle

\newcommand{\dt}{d_\tau}

\newcommand{\de}{\mathcal{D}}
\newcommand{\ct}{\mathcal{CT}}
\newcommand{\dct}{\mathcal{DT}}
\newcommand{\ec}{\mathcal{EC}}

\section{Introduction}


A hypergraph is a pair $(V,E)$ where $V$ is a set, called a \emph{vertex set} of the hypergraph and $E$ is a family of subsets of $V$, whose elements are called \emph{edges} of the hypergraph.
A hypergraph is $n$-\emph{uniform} if every of its edges contains exactly $n$ vertices.
In a fixed hypergraph, the \emph{degree of a vertex} $v$ is the number of edges containing $v$,
the \emph{degree of an edge} $e$ is the number of other edges of the hypergraph which have nonempty intersection with $e$.
The maximum edge degree of hypergraph $H$ is denoted by $\Delta(H)$.

An $r$-coloring of hypergraph $H=(V,E)$ is a mapping from the vertex set $V$ to the set of $r$ colors, $\{0,\ldots,r-1\}$.
A coloring of $H$ is called \emph{proper} if it does not create monochromatic edges (i.e. every edge contains at least two vertices which receives different colors).
A hypergraph is said to be $r$-\emph{colorable} if there exists a proper $r$-coloring of that hypergraph.
Finally, the chromatic number of hypergraph $H$ is the minimum $r$ such that $H$ is $r$-colorable.

\bigskip
In general, the problem of deciding whether given uniform hypergraph is $r$-colorable is NP-complete.
So, it is natural to investigate easily checked conditions which guarantee $r$-colorability.
In the current paper we concentrate on establishing such conditions, for simple hypergraphs and related systems, in terms of restrictions for the maximum edge degree.

\bigskip
The first quantitative relation between the chromatic number and the maximum edge degree in uniform hypergraph was obtained by Erd\H{o}s and Lov\'asz in their classical paper \cite{ErdLov}.
They proved that if $H$ is an $n$-uniform hypergraph and
\begin{equation}\label{bound:erdlov}
  \Delta(H)\leq \frac 14r^{n-1},
\end{equation}
then $H$ is $r$-colorable.
The result was derived by using Local Lemma, which first appeared in the same paper and since that time became one of the main tools of extremal and probabilistic combinatorics.

\bigskip
However the bound \eqref{bound:erdlov} was not tight.
The restriction on the maximum edge degree was successively improved in a series of papers.
We mention only the best currently known result, the reader is referred to the survey \cite{RaigShab} for the detailed history.

In connection with the classical problem related to Property B, Radhakrishnan and Srinivasan \cite{RadhSrin} proved that any $n$-uniform hypergraph $H$ with
\begin{equation}\label{bound:radhsrin}
  \Delta(H)\leq 0,17 \sqrt{\frac n{\ln n}}2^{n-1}
\end{equation}
is $2$-colorable.
Their proof was based on a clever random recoloring procedure with application of Local Lemma.

Recently a  generalization of the result \eqref{bound:radhsrin} was found by Cherkashin and Kozik \cite{CherKozik}.
They showed that, for a fixed $r\geq 2$ there exists a positive constant $c(r)$ such that for all large enough $n>n_0(r)$, if $H$ is an $n$-uniform hypergraph and
\begin{equation}\label{bound:cherkozik}
  \Delta(H)\leq c(r)\left(\frac{n}{\ln n}\right)^{\frac {r-1}{r}}r^{n-1},
\end{equation}
then $H$ is $r$-colorable.
In the case of two colors the proof from \cite{CherKozik} gives the same result as in \eqref{bound:radhsrin}, but it is shorter and easier.

\bigskip
Extremal problems concerning colorings of hypergraphs are closely connected to the classical questions of Ramsey theory (e.g. to find quantitative bounds in Ramsey Theorem or Van der Waerden Theorem).
The hypergraphs appearing in these challenging problems are very close to be simple.
Recall that hypergraph $(V,E)$ is called \emph{simple} if every two of its distinct edges share at most one vertex, i.e. for any $e,f\in E$, $e\ne f$,
\[
  |e\cap f|\leq 1.
\]
It is natural to expect that it is easier to color simple hypergraphs and that the bounds \eqref{bound:erdlov}--\eqref{bound:cherkozik} can be improved.

The first Erd\H{o}s--Lov\'asz--type result for simple hypergraphs was obtained by Szab\'o \cite{Szabo}.
He proved, that for arbitrary $\varepsilon>0$, there exists $n_0=n_0(\varepsilon)$ such that for any $n>n_0$ and any $n$-uniform simple hypergraph $H$ with maximum vertex degree at most $2^n n^{-\varepsilon}$, the chromatic number of $H$ equals two.
This theorem was extended to the analogous statements concerning edge degrees and to the arbitrary number of colors by Kostochka and Kumbhat \cite{KostKumb}.
They proved that for arbitrary $\varepsilon>0$ and $r\geq 2$, there exists $n_0=n_0(\varepsilon,r)$ such that if $n>n_0$ and an $n$-uniform simple hypergraph $H$ satifies
\begin{equation}\label{bound:kk}
  \Delta(H)\leq n^{1-\varepsilon}r^{n-1},
\end{equation}
then $H$ is $r$-colorable.
Since $\varepsilon>0$ is arbitrary in \eqref{bound:kk} then, of course, it can be replaced by some infinitesimal function $\varepsilon=\varepsilon(n)>0$, for which  $\varepsilon(n)\to 0$ as $n\to\infty$.
Few paper were devoted to the problem of estimating the order of its growth.
Kostochka nad Kumbhat themselves stated that $\varepsilon(n)=\Theta(\log\log\log n/\log\log n)$ (all logarithms are natural).
In \cite{Shab} Shabanov showed that one can take $\varepsilon(n)=\Theta(\sqrt{\log\log n/\log n})$.
Recently further progress was made independently by Kozik \cite{Kozik} and by Kupavskii and Shabanov \cite{KupShab}, who proved respectively that
bounds
\[
  \Delta(H)\leq c\, \frac {n}{\log n}r^{n-1} \quad \text{and} \quad   \Delta(H)\leq c\, \frac {n \log\log n}{\log n}r^{n-1}
\]
guarantee $r$-colorability of simple $n$-uniform hypergraphs.

\bigskip
The main result of the current paper completely removes factor $n^{-\varepsilon}$ from the bound \eqref{bound:kk}.

\begin{theorem}
\label{thm:simp}
There exists a positive constant $\alpha$ such for every $r\geq 2$, and every $n\geq 3$, any simple $n$-uniform hypergraph with maximum edge degree at most $\alpha\cdot n\; r^{n-1}$ is $r$-colorable.
\end{theorem}

Note that in comparison with  \eqref{bound:kk},  Theorem \ref{thm:simp} holds for any $r\geq 2$, not only for fixed values of $r$.
However, when $r$ becomes large, better bounds can be given.
In particular Frieze and Mubayi \cite{FrMub} showed that if $H$ is an $n$-uniform simple hypergraph with
\[
  \Delta(H)\leq c(n)r^{n-1}\ln r,
\]
where $c(n)>0$ is a function of $n$, then $H$ is $r$-colorable.
It follows form the proof in \cite{FrMub} that $c(n)=\Theta(n^{2-2n})$, so their result becomes nontrivial only for large values of $r$.

\bigskip
The restriction on the maximum edge degree in Theorem \ref{thm:simp} is not far  from the best possible.
In \cite{KR} Kostochka and R\"odl proved that, for any $n,r\geqslant 2$, there exists an $n$-uniform  simple hypergraph with
$$
  \Delta(H)\leqslant n^2r^{n-1}\ln r,
$$
which is not $r$-colorable.
Therefore the gap between two bounds is of the order $n\ln r$.

\bigskip
Methods used in the proof of Theorem \ref{thm:simp} can be used to address analogous problems in other classes of graphs.
We present such an extension concerning hypergraphs of arithmetic progressions over integers.
That allows us to derive a new lower bound for Van der Waerden numbers.
Van der Waerden number $W(n,r)$ is the minimum $N$ such that in any $r$-coloring of integers $\{1, \ldots, N\}$ there exists a monochromatic arithmetic progression of length $n$.
\begin{theorem}
\label{thm:VdW}
There exists positive $\beta$ such that for every $r\geq 2$ and $n\geq 3$, we have
\[
	W(n,r) \geq \beta r^{n-1}.
\]
\end{theorem}
That improves over the bound of Szab\'{o} of the order $n^{o(1)}r^{n-1}$ and over recent bounds by Kozik \cite{Kozik} and by Kupavskii and Shabanov \cite{KupShab}, who proved respectively that
\[
	W(n,r) \geq \beta \frac{1}{\log n}r^{n-1} \quad \text{and} \quad 	W(n,r) \geq \beta \frac{\log\log n}{\log n} r^{n-1},
\]
for some constants $\beta$.
When $r=2$, better bounds for some $n$ can be derived from the bound of Berlekamp \cite{Berlekamp}, which states that $W(p+1,2) \geq p2^p$, for prime $p$.

\bigskip
The rest of the paper is organized as follows.
In the next section we introduce notation and discuss a special variant of Local Lemma.
In Section \ref{sec:MGC} we describe and analyze the coloring algorithm.
Theorem \ref{thm:simp} is proved in Section \ref{sec:simp}.
In Section \ref{sec:VdW} we deduce a lower bound for Van der Waerden numbers.
Finally Section \ref{sec:corollaries} is devoted to various corollaries.

\section{Notations and tools}


\subsection{Trees}
Trees considered in this paper are always rooted.
Vertices of trees are called \emph{nodes}.
For every two adjacent nodes, the node that is farther from the root of the tree is a \emph{child} of the other node.
It is convenient to use the following definition of subtree.
For a tree $T$, \emph{a subtree} is a connected subgraph $T'$ such that $T$ after removing nodes of $T'$ remains connected and there exists $x$ in $T$ such that any path from any node of $T'$ to the root of $T$ contains $x$.
Such subtree is always rooted at the node that is of the closest distance to the root of the original tree.
Direct subtree of $T$ is a subtree rooted at a child of the root of $T$. 
For a fixed tree, by $d(x)$ we denote the number of children of the node $x$ in the tree.

\subsection{Local Lemma}
In \cite{Szabo} Szab\'{o} used a specific variant of Local Lemma derived from the general version by Beck in \cite{Beck}.
We use the following generalization from \cite{Kozik} of the  Beck's variant.

\begin{lemma}
\label{lm:local}
	Let $\vS=\{X_1, \ldots, X_m\}$ be independent random variables (or vectors) in arbitrary probability space and let $\evS$ be a finite set of events determined by these variables.
For $A\in \evS$, let $\vars (A)$  be the minimum set of variables that determines $A$.
For $X\in \vS$, we define formal polynomial $w_X(z)$ in the following way:
\[
	w_X(z)= \sum_{A\in \evS: X\in \vars(A)} \prob(A) \; z^{|\vars(A)|}.
\]
Suppose that a polynomial $w(z)$ \emph{dominates} all polynomials $w_X(z)$ for $X\in \vS$ i.e. for every  real $z_0 \geq 1$ we have $w(z_0) \geq w_x(z_0)$.
If there exists $\tau_0\in (0,1)$ such that
\[
	w\left(\frac{1}{1-\tau_0}\right) \leq \tau_0,
\]
then all events from $\evS$ can be simultaneously avoided with positive probability, i.e. $\prob\left(\bigcap_{A\in\evS}\overline{A}\right)>0$.
\end{lemma}

We say that $w_X(z)$ is a \emph{local polynomial for random variable $X$} and $w(z)$ is a \emph{local polynomial}.

All applications of the above lemma within the current paper concern $n$-uniform hypergraphs.
For such hypergraph, we always choose $\tau_0=\tau_0(n)=1/n$.
For convenience we put $z_0=z_0(n)= \frac{1}{1-\tau_0(n)}$.


\newcommand{\ic}{c^0}
\section{Coloring algorithm}
\label{sec:MGC}
Let $H=(V, E)$ be an $n$-uniform  hypergraph and $r$ be a number of colors.
We present and analyse an algorithm that tries to improve coloring given on the input.
The algorithm is parameterized by $p \in (0, 1/2)$ and gets two inputs: first is an \emph{initial coloring} $\ic$ of the hypergraph, second is an injective  function $\sigma:V \to [0,1]$, called \emph{weight} assignment.
For every vertex $v$, color $\ic(v)$ assigned to it by the initial coloring is called \emph{the initial color} of $v$.
We say that vertex $v$ is a $j$-vertex if it is initially colored with $j$.
The value $\sigma(v)$ is called the \emph{weight} of $v$.
Vertex $v$ is called \emph{free} if $\sigma(v) \leq p$.
Recall that an edge is monochromatic w.r.t. some coloring if all its vertices gets the same color.
In any set of vertices the \emph{first vertex} is the vertex $v$ with minimum weight, i.e. minimum value of $\sigma(v)$.
We use a succinct notation $(n)_r$ to denote the value of $n \pmod r$.
\\
\begin{algorithm-hbox}[H]
\caption{$r$--coloring of hypergraph $(V,E)$}\label{alg:MGC}
	\textbf{Input:} $c:V \to \{0, \ldots, r-1 \},\; \sigma:V \to (0,1]$ injective\\
 	\While{there exists a monochromatic edge whose first non--recolored vertex $v$ is free}{
 		$c(v) \gets (c(v)+1)_r$\hspace{1cm} (i.e. $v$ is recolored with $(c(v)+1)_r$)
 	} 	
\Return $c$
\end{algorithm-hbox}	
Note that during the evaluation of the algorithm every vertex changes its color at most once, therefore the procedure always stops.

An edge $f \in E$ is called \emph{degenerate} if it contains at least $n/2$ free vertices.
It is said to be \emph{dangerous} if there exists $i\in \{0,\ldots,r-1\}$, called \emph{dominating color} of $f$, such that all non--free vertices of $f$ are initially colored with $i$ and every free vertex of $v$ is initially colored with $i$ or $(i-1)_r$.
The remaining edges are called \emph{safe}.
It is straightforward to check that safe edges are never monochromatic in the coloring returned by the algorithm.

\bigskip
In order to analyse the situations in which the coloring returned by the algorithm is not proper we introduce the following definitions.
An \emph{h-tree} is a rooted tree labelled according to the following rules:
\begin{enumerate}
	\item each node $x$ is labelled by an edge $e(x)$ of the hypergraph $H$,
	\item each edge $f$ is labelled by a vertex $v(f)$ of the hypergraph $H$,
	\item for an edge $f=(x_1, x_2)$ we have $e(x_1) \cap e(x_2) \ni v(f)$.
\end{enumerate}
An h-tree is called \emph{disjoint} if for every two distinct nodes $x,y$, edge $e(x)$ intersects $e(y)$ in at most one vertex and the intersection is not empty only when $x,y$ are adjacent in the tree (in particular the set of labels of the vertices forms a hypertree).
A node $x$ \emph{contains} vertex $v$ if $v\in e(x)$.
Similarly, we say that an h-tree \emph{contains} edge $f$ if some node of the tree is labelled with $f$.
To make a clear distinction, rooted trees without any labellings are called \emph{structures}.

\bigskip
An h-tree of only one node $x$ is called \emph{alternating} (w.r.t. some weight assignment $\sigma$ and initial coloring $\ic$),  if $e(x)$ is neither degenerate nor safe.
A tree with root $x$ and direct subtrees $t_0, \ldots, t_{k-1}$ is \emph{alternating} if
\begin{enumerate}
	\item $e(x)$ is neither degenerate nor safe, let $i$ be the dominating color of $e(x)$,
	\item all subtrees  $t_0, \ldots, t_{k-1}$ are alternating, let $y_0, \ldots, y_{k-1}$ be their roots,
	\item for every $j\in \{0, \ldots, k-1\}$, the dominating color of $e(y_j)$ is $(i-1)_r$,
	\item for every $j\in \{0, \ldots, k-1\}$, the first $(i-1)_r$-vertex of $e(y_j)$ belongs to $e(x)$,
	\item every free $(i-1)_r$-vertex of $e(x)$ is the first  $(i-1)_r$-vertex of some $e(y_j)$.
\end{enumerate}

An alternating h-tree is \emph{downward complete} if all the edges labelling its leaves are monochromatic in the initial coloring, i.e. every edge, that labels a leaf, with some dominating color $i$,  does not contain a free $(i-1)_r$-vertex.
It is \emph{complete} if additionally the root does not contain any free vertex initially colored with its dominating color.

\begin{prop}
\label{prop:cc2}
If for initial coloring $\ic:V\to \{0,\ldots,r-1\}$ and injective $\sigma : V \to [0, 1]$, there are no degenerate dangerous edges in the hypergraph and the algorithm produces a coloring which is not proper, then there exists a complete h-tree (w.r.t. $\sigma$ and $\ic$).
\end{prop}
\begin{proof}
Whenever during the evaluation of the algorithm some vertex $v$ is recolored, it is the first non-recolored free vertex of some edge $f$ that at that moment was monochromatic.
In this case vertex $v$ is said to \emph{blame} edge $f$ and for every free vertex, which has been recolored during the evaluation of the algorithm, we choose one edge to be blamed.
We say that edge $f_1$  blames edge $f_2$ if $f_1$ contains a vertex that blames $f_2$.
Note that only dangerous edges can be blamed.

Relation of blaming defines a directed graph $B$ on the edges of the hypergraph.
This graph turns out to be acyclic. Indeed let $f_1, \ldots, f_k$ be a directed cycle in this graph (i.e. $f_j$ blames $f_{j+1}$ and $f_k$ blames $f_1$) and suppose that $f_1$ is the edge of the cycle that became blamed last during the evaluation of the algorithm.
Let $i_1, \ldots, i_k$ denote the dominating colors of $f_1, \ldots, f_k$.
Clearly we have $i_j=(i_{j+1}+1)_r$ and $i_k=(i_{1}+1)_r$.
Edge $f_1$ became blamed last and until that time the first vertex of $f_1$ has not been recolored.
That vertex must belong also to $f_k$ which shows that until that time $f_k$ contained a vertex of color $(i_k-1)_r$.
Therefore this edge could not have been monochromatic.
That implies that it cannot be blamed.
This contradicts the assumption that $f_1$ has been blamed last.

Suppose that the algorithm constructed a coloring which is not proper.
Therefore there exists a monochromatic edge $f$.
Let $T$ be the set of directed paths in $B$ starting from $f$.
We say that two paths in $T$ are adjacent if one path is a prefix of the other and is shorter by exactly one.
Clearly $T$ with such defined adjacency is a tree, from now the elements of $T$ are called nodes.

Every node $x$ of $T$ is a path of edges of the hypergraph, we choose a label $e(x)$ as the last edge of the path.
Then for every adjacent nodes $x_1, x_2$ we know that $e(x_1)$ blames $e(x_2)$ (or the other way around.)
In particular their intersection is nonempty.
Moreover, the first vertex $w$ of $e(x_2)$ belongs to $e(x_1)$.
We set the label of tree edge $(e(x_1), e(x_2))$ to $w$.
Such labelled $T$ is a complete h-tree.
\end{proof}

\begin{corollary}
\label{cor:condT}
If there are no degenerate dangerous edges and no complete h-trees w.r.t. some weight assignment  and initial coloring, then the algorithm produces a proper coloring given that assignment and coloring on the input.
\end{corollary}


\newcommand{\cd}{C}

\section{Simple hypergraphs}
\label{sec:simp}

\subsection{Some auxiliary claims.}

We start with estimating the number of h-trees in a hypergraph with bounded edge degrees.
\begin{prop}
\label{prop:treesNb}
Let $H=(V,E)$ be a hypergraph with maximum edge degree $D$ and let $v\in V$ be its arbitrary vertex. Then the number of disjoint h-trees of size $N$ containing $v$ is at most $(4D)^N$.
\end{prop}
\begin{proof}
Let us fix some specific tree structure $s$ of size $N$.
We have $N$ possible choices for the node $x$ containing $v$, and at most $D$ possible labels for that node.

Starting from this we extend labelling according to the following rule: for every unlabelled node $y$ which is adjacent to a labelled node $x$ pick any edge that intersects $e(x)$ as a label.
Each time we have at most $D$ choices for the next label, hence the total number of node labellings constructed in such a way is at most $D^N$.

When all the nodes are labelled properly (i.e. no two distinct nodes $x,y$ have the size of intersection $e(x)\cap e(y)$ larger than 1), we can uniquely extend the labelling to edges obtaining an h-tree with structure $s$.

Clearly every such disjoint h-tree  with structure $s$ containing $v$ in $x$ can be constructed in this way.
Therefore $v$ belongs to at most $N D^N$ h-trees with structure $s$. 	
The number of possible structures of size $N$ does not exceed  $4^N/N$.
Hence the total number of disjoint h-trees containing $v$ of size $N$ is smaller than $(4D)^N$.
\end{proof}

The second type of structures that play important role in our proof are cycles.
A sequence of distinct edges $(f_0, \ldots, f_{k-1})$ of a hypergraph forms a \emph{simple cycle} if for every $i\in \{0, \ldots, k-1\}$ edge $f_i$ intersects only edges $f_{(i-1)_k}$ and $f_{(i+1)_k}$.
The next proposition estimates the number of simple cycles in simple hypergraphs.

\begin{prop}
	\label{prop:cycleNb}
Let $H=(V,E)$	be a simple $n$-uniform hypergraph with maximum edge degree $D$ and let $v\in V$ be its arbitrary vertex.
Then the number of cycles of length $N\geq 2$ containing $v$ is at most $N D^{N-1} n^2$.
\end{prop}
\begin{proof}
We have $N$ possible choices for the index $j$ of the edge that contains $v$ in the cycle (which is formally a sequence).
Then there are at most $D$ possible choices for edge $f_j$.
Suppose that edges $f_j, \ldots, f_{(j+s)_N}$ are already chosen.
If $s<N-2$ then we have at most $D$ choices for the edge  $f_{(j+s+1)_N}$(it has to intersect $f_{(j+s)_N}$).
When choosing the last edge we must ensure that it intersects both $f_j$ and $f_{(j-2)_N}$.
The number of such edges is at most $n^2$  (we have $n$ choices for each vertex from the intersections and once they are fixed there exists at most one edge containing them both since the hypergraph is simple).
Altogether it gives at most $N D^{N-1} n^2$ different cycles.
\end{proof}

Now we are going to estimate the probabilities that given h-tree in a simple $n$-uniform hypergraph becomes complete under the random input.
Note that, in simple hypergraphs, any correct labelling of the nodes of an h-tree uniquely determines the labels of the edges.

\begin{prop}
\label{prop:disjProb}
Let $H$ be an $n$-uniform hypergraph and let weight assignment function $\sigma$ and initial $r$-coloring $\ic$ of $H$ be chosen uniformly at random.
Then the probability that a fixed disjoint h-tree $T$ of size $N$ is downward complete is smaller than
\[
  r^{-(n-1)N} \left( \frac{2}{n} \right)^{N-1}
\]
and the probability that it is complete is smaller than
\[
  (1-p)^{n/2} r^{-(n-1)N} \left( \frac{2}{n} \right)^{N-1}.
\]
\end{prop}
\begin{proof}
Note that if the maximum degree of $T$ is larger than $n/2$ then $T$ can not be downward complete (since it can not contain degenerate edges).
Let $x$ be the root of $T$ and let $T_1,\ldots,T_k$ be direct subtrees of the root with roots $y_1,\ldots,y_k$.
Suppose that $T$ is downward complete and the dominating color of $e(x)$ is $i$.
Then it is necessary that:
\begin{enumerate}
	\item all trees $T_1,\ldots,T_k$ are downward complete;
	\item for every $j\in [k]$, the dominating color of $e(y_j)$ is $(i-1)_r$;
	\item the vertices of $e(x)$ colored initially with $(i-1)_r$ are exactly the vertices $v(x,y_1),\ldots,v(x,y_k)$;
	\item for every $j\in [k]$, vertex $v(x,y_j)$ is the first vertex of $e(y_j)$ initially colored with $(i-1)_r$.
\end{enumerate}

So, for every node $y$ of a downward complete h-tree, which is not the root, some specific set of $n-d(y)$ vertices of $e(y)$ are initially colored with specific color, in fact, dominating color of $y$, and one specific vertex among these vertices must be the first in the weight order.
This is the unique common vertex of the edge and its parent in the tree.

Moreover, if $i$ is the dominating color of the root $x$ of the tree, then $e(x)$ must have exactly $n-d(x)$ vertices colored initially with $i$.
All these events concern pairwise disjoint sets of vertices so they are all independent.
The probability that they all hold is
\begin{align*}
	 \left(\frac{1}{r}\right)^{n-d(x)} \prod_{y\in T, y\neq x} \frac{1}{(n-d(y)) r^{n-d(y)}}	 \\
	=  r^{-nN} \left(\prod_{y\in T, y\neq x} \frac{1}{n-d(y)} \right) \left( \prod_{y\in T}  r^{d(y)} \right) \\
	= r^{-nN} \left(\prod_{y\in T, y\neq x} \frac{1}{n-d(y)} \right) r^{\sum_{y\in T}d(y)}.
\end{align*}
Since $n-d(x)>n/2$ and the sum of degrees in a tree is the number of its vertices minus 1 we get upper bound
\begin{align*}
	\leq  r^{-nN} \left( \frac{2}{n} \right)^{N-1} r^{N-1}.
\end{align*}
Finally there are $r$ choices for the dominating color of the root so the total probability that the h-tree is downward complete is smaller than $r^{-N(n-1)} \left( \frac{2}{n} \right)^{N-1}$.

The root $x$ of complete h-tree with dominating color $i$ additionally must not contain any free vertices initially colored with $i$. Since $d(x)<n/2$ it must have at least $n/2$ such vertices.
Therefore the probability that $T$ is complete is at most
\[
	(1-p)^{n/2} r^{-(n-1)N} \left( \frac{2}{n} \right)^{N-1}.
\]
\end{proof}

Now we are ready to finish the proof of Theorem \ref{thm:simp}.

\subsection{Proof of Theorem \ref{thm:simp}}
Let $H=(V,E)$ be a simple $n$-uniform hypergraph with maximum edge degree $D\leq \alpha\cdot nr^{n-1}$.
We are going to apply Lemma \ref{lm:local} to prove that for suitably chosen constant $\alpha$ and $p=p(n)$, Algorithm \ref{alg:MGC} succeeds with positive probability, when the input is chosen uniformly at random.
For that reason we specify four families of events to be avoided and analyse their contribution to the local polynomial.

The first kind of event that we want to avoid is that an edge is degenerate and dangerous.
If there are no degenerate dangerous edges, then by Corollary \ref{cor:condT} it is sufficient to avoid
the event that there exists an h-tree which is complete.
For disjoint h-trees, this is exactly the situation that we avoid.
If an h-tree is not disjoint it may still contain large downward complete disjoint subtree.
Such subtrees of size $\geq \log(n)$ are large enough to be avoided and this is the third kind of events that we want to avoid.

\bigskip	
Suppose that a complete h-tree $T$ is not disjoint and does not contain downward complete disjoint subtree of size $\geq \log(n)$.
Let $T'$ be a not disjoint subtree of $T$ with minimal size. 
Clearly, any direct subtree of $T'$ is disjoint and has size $<\log(n)$.
Since $T'$ is not disjoint there exists a shortest path in this tree $x_1,\ldots,x_k$ of length $< 2\log(n)$ for which $e(x_1)$ intersects $e(x_k)$.
In particular, the sequence  $e(x_1), \ldots, e(x_k)$ forms a simple cycle in the hypergraph.

Nodes $x_1, \ldots, x_k$ form a path in a complete h-tree, therefore for every $j\in [k-1]$ the dominating color of $e(x_j)$ differs from the dominating color in $e(x_{j+1})$ by one.
Every cycle with that property will be called \emph{bad}.
Note that a bad cycle has length at least $3$.
Bad cycles are the fourth kind of events that we consider.

\bigskip	
By the above discussion, if there are no degenerate dangerous edges, disjoint complete h-trees, disjoint downward complete h-trees of size $\geq \log(n)$ and bad cycles of length $< 2\log(n)$, then there are no complete h-trees.
Let $\de_f, \ct_T, \dct_S, \ec_C$ be the events that edge $f$ is degenerate, disjoint h-tree $T$ is complete, disjoint h-tree $S$ is downward complete, cycle $C$ is bad.
	
In the following subsections we analyse contribution to the local polynomial for each of these kinds of events.	

\subsubsection{Events $\de$.}
Every vertex belongs to at most $D$ edges and the probability that an edge is degenerate and dangerous is smaller than $r(2/r)^n\binom{n}{n/2} p^{n/2}$.
Indeed, if the edge is dangerous then it is two-colored in the initial coloring and has at least $n/2$ free vertices.
Hence the contribution of this kind of events to the local polynomial is at most
\begin{align*}
	w_\de(z)&= D\, 2^nr^{1-n}\binom{n}{n/2} p^{n/2} z^n.
\end{align*}


\subsubsection{Events $\ct$.}
By Proposition \ref{prop:treesNb} the number of disjoint h-trees of size $N$ in $H$ containing some fixed vertex is at most $(4D)^N$.
By Proposition \ref{prop:disjProb} the probability that such an h-tree is complete is smaller than $(1-p)^{n/2} r^{-(n-1)N} (2/n)^{N-1}$.
Therefore the contribution of disjoint complete trees to the local polynomial is at most:
\begin{align*}
	w_\ct(z)&= \sum_{N \in \Nat_1}
	(4D)^N
	\left( (1-p)^{n/2} r^{-(n-1)N} (2/n)^{N-1}	\right)
	 z^{Nn}\\
	 & = \frac{n (1-p)^{n/2}}{2} \sum_{N \in \Nat_1}
	\left( \frac{8Dz^n}{n \; r^{n-1}}   \right)^N.
\end{align*}
For $D < 16^{-1} z^{-n} n r^{n-1}$, the sum is convergent to the value which is smaller than one.
Then it suffices to set $p\geq 5\log(n)/n$ to get $w_\dct(z) \leq 1/n^{3/2}$.

\subsubsection{Events $\dct$.}
By Proposition \ref{prop:disjProb} the probability that a disjoint h-tree of size $N$ in $H$ is downward complete is smaller than $r^{-(n-1)N} (2/n)^{N-1}$.
Therefore the contribution of disjoint downward complete trees of size at least $\log(n)$ to the local polynomial is at most:
\begin{align*}
	w_\dct(z)&= \sum_{N \geq \log(n)}
	(4D)^N
	\left( r^{-(n-1)N} (2/n)^{N-1}	\right)
	 z^{Nn}\\
	 & = \frac{n}{2}
	\left( \frac{8Dz^n}{n \; r^{n-1}}   \right)^{\lceil \log(n) \rceil}	
	  \sum_{N \in \Nat}
	\left( \frac{8Dz^n}{n \; r^{n-1}}   \right)^N.
\end{align*}
Again, for $D < 16^{-1} z^{-n} n r^{n-1}$, the sum is convergent and bounded by two.
If additionally $D< 8^{-1}e^{-5/2} z^{-n} n r^{n-1}$, then
	 \[
	 	\left( \frac{8Dz^n}{n \; r^{n-1}}   \right)^{\lceil \log(n) \rceil}	
	 	< n^{-5/2},
	 \]
so $w_\dct(z) < 1/n^{3/2}$.
	
\subsubsection{Events $\ec$.}	
The number of cycles of length $N$ containing specific vertex is at most $N D^{N-1} n^2$ (Proposition \ref{prop:cycleNb}).
Let $f_0, \ldots f_{N-1}$ be a bad cycle and let $i_0, \ldots, i_{N-1}$ denote dominating colors of consecutive edges of the cycle.
Let $v_0, \ldots, v_{N-1}$ be such that for $j\in \{0, \ldots, n-1\}$ we have $f_j \cap f_{(j+1)_N}=\{v_j\}$.
For the cycle to be bad, it is necessary that for every $j\in \{0, \ldots, N-1\}$, any vertex of $f_j\setminus \{v_j\}$ either has initial color $i_j$ or has initial color $(i_j-1)_r$ and is free.
That necessary condition holds for every edge $f_0, \ldots, f_{N-1}$ independently with probability $\left((1+p)/r\right)^{n-1}$.
Finally we have $r$ possibilities for the dominating color of $f_1$ and at most two possibilities for dominating colors of each $f_2, \ldots, f_{N-1}$.
Total probability of being bad for a cycle of length $N$ is smaller than
\[
	r 2^{N} \left(\frac{1+p}{r}\right)^{(n-1)N}.
\]
Therefore the contribution to the local polynomial from the bad cycles of length at most $2\log(n)$ does not exceed
\begin{align*}
	w_\ec(z)&= \sum_{ 3 \leq N  <2 \log(n)}
	\left( N D^{N-1} n^2 \right)
	\left( r 2^{N} \left(\frac{1+p}{r}\right)^{(n-1)N} \right)
	 z^{Nn}\\
	 & \leq \frac{4 \log(n) \;n^{2+12 \log(n)} (1+p)^{n-1} z^n}{n^6 r^{n-2} }
	  \sum_{3 \leq N <2 \log(n)}
	\left( \frac{2 (1+p)^{n-1} D z^n}{n^6 \; r^{n-1}}   \right)^{N-1}.
\end{align*}
For $D< \frac{1}{4 z^{n} (1+p)^{n-1}} n^{6} r^{n-1}$ the sum is convergent and bounded by one. Then $w_\ec(z) < \frac{4 \log(n) \;n^{2+12 \log(n)} (1+p)^{n-1} z^n}{n^6 r^{n-2} }$.

\subsubsection{Choice of the parameters}
Let us finish the proof. We set $p=5 \log(n)/n$ and $z_0=1/(1-1/n)$. Then for $D< (2e)^{-3} z_0^{-n} n r ^{n-1}$, we have

\[
	w_\ct(z_0) < 1/n^{3/2}, \:\:\: w_\dct(z_0) < 1/n^{3/2}
\]
and
\begin{align*}
	w_\de(z_0)&= D \frac {2^n}{r^{n-1}}\binom{n}{n/2} \left( \frac{5 \log(n)}{n}\right)^{n/2} z_0^n <
	\left( \frac{80 \log(n)}{n}\right)^{n/2} n.
\end{align*}
The bound for $w_\de(z_0)$ is super-exponentially small in $n$.

\bigskip
We have $(1+p)^{n-1} < n^5$ so our bound for $D$ implies that $D< \frac{1}{4 z_0^{n} (1+p)^{n-1}} n^{6} r^{n-1}$.
Therefore
\begin{align*}
	w_\ec(z_0) < \frac{4 \log(n) \;n^{2+12 \log(n)} (1+p)^{n-1} z_0^n}{n^6 r^{n-2} }
\end{align*}
which is exponentially small in $n$ (recall $r\geq 2$).

Therefore for large enough $n$, values of all these polynomials are bounded by $1/n^{3/2}$ which, by Lemma \ref{lm:local}, implies by  that all events of types $\de, \ct, \dct, \ec$ can be simultaneously avoided.
By Corollary \ref{cor:condT} it implies that for all large enough $n$, $H$ is $r$-colorable.
Note that $z_0^{-n}\sim e^{-1}$, so it is enough to choose $\alpha = (1-\epsilon) 2^{-3} e^{-4}$ for any positive $\epsilon<1$.

We proved that there exists $n_0$ such that for any $n>n_0$, any $r\geq 2$, arbitrary  simple $n$-uniform  hypergraph with maximum edge degree at most $(2e)^{-4} n r^{n-1}$ is $r$-colorable.
Thus, there exists $\alpha>0$ such that, for all $n\geq 3$, any such hypergraph with maximum edge degree at most $\alpha\cdot n r^{n-1}$ is $r$-colorable.
Theorem \ref{thm:simp} is proved.

\section{Van der Waerden numbers}
\label{sec:VdW}

The main aim of this section is to obtain a new lower bound for the Van der Waerden function $W(n,r)$.


For fixed $n$ and $M$, let $H_{(n,M)}$ denote the hypergraph of arithmetic progressions with vertex set $[M]=\{1,\ldots,M\}$ and edge set consisting of all arithmetic progressions of length $n$ contained in $[M]$.
Clearly, $H_{(n,M)}$ is $r$-colorable iff $W(n,r) > M$.

We are going to prove that for appropriately chosen $\beta$ and $M\leq  \beta r^{n-1}$, Algorithm \ref{alg:MGC} colors properly $H_{(n,M)}$ with $r$ colors, with positive probability, when the initial coloring and the weight assignment are chosen uniformly at random.
We use the following simple facts concerning hypergraph $H_{(n,M)}$.

\begin{prop}
\label{prop:arith_prog}
(i) The maximum vertex degree of $H_{(n,M)}$ is at most $M$.

(ii) The 2-codegree of $H_{(n,M)}$ is at most $n^2$, so for every edge $f$ there exists at most $n^4/2$ other edges $f'$ such that $|f\cap f'|\geq 2$.

(iii) For any two vertices $v_1,v_2$ of $H_{(n,M)}$, there are at most $(3n/2)^2 $ pairs of edges $f_1, f_2$ for which $v_1\in f_1, v_2, \in f_2$ and $|f_1 \cap f_2 |> n/2$.
\end{prop}
\begin{proof}
(i) Every progression containing a fixed vertex $v$ is uniquely defined by a position of the vertex in the progression and by a difference of the progression.

(ii) Every $f$, containing fixed vertices $v$ and $v'$, is uniquely defined by positions of $v$ and $v'$.

(iii) In this case $f$ and $f'$ form a longer arithmetic progression of length at most $3n/2$.

\end{proof}

From now on we denote the maximum edge degree of $H_{(n,M)}$ by $D$ and focus on hypergraphs of arithmetic progressions of maximum edge degree $D$.
It follows from Proposition \ref{prop:arith_prog} that $D<nM$.

\bigskip
Just like in the proof of Theorem \ref{thm:simp} we are going to prove that  with positive probability there are no degenerate dangerous edges, complete disjoint h-trees, downward complete disjoint h-trees of size at most $\log(n)$ and bad cycles of length at most $2 \log(n)$ (with appropriately redefined bad cycles).
The analysis of the first three events is the same as in the case of simple hypergraphs, because it does not need the considered hypergraph to be simple.
In the proof of Theorem \ref{thm:simp} the only events in which we used the simplicity of the hypergraph were bad cycles.
Below we present an alternative analysis of these events which is valid for hypergraphs of arithmetic progressions.

\bigskip
The choice of the parameters in the proof remains the same, i.e. we set
\[
  p=\frac{5 \log(n)}n,\;\; z_0=1/(1-1/n),\;\; \mbox{and consider that }D< (2e)^{-3} z_0^{-n} n r^{n-1}.
\]
In particular we can choose $\beta=(1-\epsilon) 2^{-3} e^{-4}$ for any positive $\epsilon < 1$.

\subsection{Proof of Theorem \ref{thm:VdW} (patch to the proof of Theorem \ref{thm:simp})}

Let $T$ be a complete h-tree which is not disjoint and does not contain disjoint downward complete subtree of size $\geq \log(n)$.
Let $T'$ be its minimal subtree that is not disjoint. Clearly any subtree of $T'$ is disjoint and smaller than $\log(n)$.

Tree $T'$ is not disjoint therefore there exist distinct non-adjacent nodes $x,y$ in $T$ such that $e(x)\cap e(y)\neq \emptyset$ or there exist adjacent nodes $x,y$ with $|e(x)\cap e(y)| \geq 2$.
Let us choose nodes $x,y$ satisfying one of the above condition such that the path in $T'$ from $x$ to $y$ is the shortest. Let $x_1=x, \ldots, x_k=y$ be the consecutive vertices of that path.

In particular case when $k=2$, we have $|e(x)\cap e(y)| \geq 2$, and such pair of edges forms a cycle of length 2.

If $k\geq 3$, then the sequence $(e(x_1),\ldots, e(x_k))$  is a cycle in the hypergraph $H_{(n,M)}$ (i.e. only consecutive edges and $(e(x_1), e(x_k))$ have nonempty intersections).
Then for every $j\in [k-1]$ we have $|e(x_j) \cap e(x_{j+1})|= 1$.
The main difference in comparison with the case of simple hypergraph is that the intersection $e(x_1) \cap e(x_{k})$ can be large.

\bigskip
We redefine \emph{bad} cycle as a cycle $(e_1,\ldots,e_k)$ of length $k\geq 3$ in which every intersection of edges $e_j, e_{j+1}$ has cardinality one, no edge is easy or degenerate and the dominating colors in consecutive edges differs by exactly one $(\mod r)$.

A cycle $f_1, f_2$ of length 2 is called \emph{bad} if $|f_1\cap f_2| \geq 2$, both edges are neither degenerate nor easy and their dominating colors differ by exactly one.
In particular it means that their intersection consists of only free vertices.

\bigskip
As a result if there are no degenerate dangerous edges, but there exists a complete h-tree which is not disjoint and which has no downward complete tree of size $\geq \log(n)$, then there must exist a bad cycle of length smaller than $2 \log(n)$.

Let $\ec$ be the event that some cycle $(e_1, \ldots , e_k)$ of length $k<2 \log(n)$, is bad. Note that if the length of the cycle is at least 3 and the cycle is bad, then it satisfies $|e_j \cap e_{j+1}|=1$, for $j\in [k-1]$.

We consider three types of bad cycles. A bad cycle is of type 0 if its length equals 2.
For $k\geq 3$, the cycle $(e_1, \ldots, e_k)$ is of type 1 if $|e_1 \cap e_k| \leq n/2$ and of type 2 otherwise.

\subsubsection{Events $\ec$ for cycles of type 0}
The probability that two edges $f_1, f_2$ with intersection of size $m\geq 2$ form a bad cycle is at most
\[
	2r \left( \frac{1+p}{r} \right)^{2n-m}.
\]
However, when cycle $(f_1, f_2)$ is bad then all the vertices of the intersection are free.
When $m$ is greater than $n/2$ then both edges are degenerate contradicting our assumption that there is no dangerous degenerate edges. Therefore $m\leq n/2$ and probability that these edges form a bad cycle is at most
\[
	2r \left( \frac{1+p}{r} \right)^{3n/2}.
\]
Every vertex belongs to at most $D$ edges, and for every edge, there are at most $n^4/2$ other edges that intersects that edge in at least two vertices.
Therefore total contribution to the local polynomial from the cycles of type 0 is smaller than
\begin{align*}
	w_{\ec}^0(z)&=
		D n^4	r \left( \frac{1+p}{r} \right)^{3n/2}
		z^{2n}\\
	& < \frac{n^{10} r z^n (1+p)^n}{r^{n/2}} \left( \frac{D z^n (1+p)^n}{n^6r^n} \right).
\end{align*}
For considered parameters, value $\left( \frac{D z_0^n (1+p)^n}{n^6r^n} \right)$ is bounded by constant, while value $\frac{n^{10} r z_0^n (1+p)^n}{r^{n/2}}$ is exponentially small in $n$.
Therefore for any $r\geq 2$ and all large enough $n$, we have $w_{\ec}^0(z_0)< 1/n^2$.

\subsubsection{Events $\ec$ for cycles of type 1}
In hypergraph $H_{(n,M)}$ every two vertices belong to at most $n^2$ common edges.
Therefore the same argument as in the proof of Proposition \ref{prop:cycleNb} can be used to establish that every vertex belongs to at most $N D^{N-1} n^4$ cycles of length $N$.

Let $(e_1, \ldots, e_N)$ be a cycle of type 1.
In a bad cycle all the edges should be dangerous.
For every edge $e_j$, let $e'_j$ be the set of vertices not contained in the previous edges.
Clearly only $e'_N$ can have size smaller than $n-1$, but since the cycle is of type 1 its size is at least $n/2$.

Let $i_j$ denote the dominating color of edge $e_j$.
Then every vertex of $e'_j$ must be either initially colored with $i_j$ or initially colored with $(i_j-1)_r$ and be free.
The probability that it happens (for a fixed choice of dominating colors) is smaller than
\[
	\left(\frac{1+p}{r} \right)^{(n-1)(N-1)+n/2}.
\]
There are at most $r 2^{N-1}$ choices for the sequence of dominating colors.
Therefore the probability that $(e_1, \ldots, e_N)$ is bad is smaller than
\[
	r 2^{N-1} \left(\frac{1+p}{r} \right)^{(n-1)(N-1)+n/2}.
\]
Hence the contribution of these event to the local polynomial is at most
\begin{align*}
	w_{\ec}^1(z)&=
	\sum_{3 \leq N < 2\log(n)}
		(N D^{N-1} n^4 )
		r 2^{N-1} \left(\frac{1+p}{r} \right)^{(n-1)(N-1)+n/2}
		z^{nN}\\
	& <
		\frac{2 \log (n) n^{12\log(n)} (1+p)^{n/2} z^n}{r^{n/2-1}}
		\sum_{3 \leq N < 2\log(n)}
		\left( \frac{2 (1+p)^{n-1}D z^n}{n^6 r^{n-1}}  \right)^{N-1}.		
\end{align*}

For the same parameters as in the proof Theorem \ref{thm:simp} (i.e. $p=5 \log(n)/n$ and $z_0=1/(1-1/n)$), the sum is convergent and bounded by one.
Value
\[
  \frac{2 \log (n) n^{12\log(n)} (1+p)^{n/2} z_0^n}{r^{n/2-1}}
\]
is exponentially small in $n$. In particular, for large enough $n$, value $w_{\ec}^1(z_0)$ is smaller than $1/n^2$.

\subsubsection{Events $\ec$ for cycles of type 2}
Let $(e_1, \ldots, e_N)$ be a cycle of type 2.
Edges $e_1$ and $e_N$ have at least $n/2$ common vertices therefore as arithmetic progressions they must have the same difference.
Any two vertices $v_1,v_2$ have at most $(3n/2)^2 $ pairs of edges $f_1, f_2$ for which $v_1\in f_1, v_2, \in f_2$ and $|f_1 \cap f_2 |> n/2$.

Hence, for any fixed vertex $v$, the number of cycles of type 2 of length $N$ containing $v$ in not the first and not the last edge is at most $(N-2)D^{N-2} n^2 (3n/2)^2$.
Let us count the number of cycles of type 2 of length $N$ containing $v$ in the first edge.
We have at most $D^{N-2}$ choices for the edges $e_1, \ldots, e_{N-2}$.
We have at most $n^4$ choices for $e_N$ since it intersects $e_1$ in more than two points and every two points belong to at most $n^2$ edges.
Finally, the last edge $e_{N-1}$ has to intersect $e_{N-2}$ and $e_N$, so we have at most $n^4$ choices for it.
It gives less than $n^{8} D^{n-2}$ cycles.
Altogether the number of cycles of type 2 of length $N < 2 \log(n)$ containing some specific vertex is smaller than
\[
	N D^{N-2} n^8.
\]
The probability that such a cycle is bad is smaller than the probability that edges $e_1, \ldots, e_{N-1}$ are all dangerous, not degenerate and dominating colors of consecutive edges differ by exactly one $(\mod r)$.
That probability is smaller than
\[
	r 2^{N-2} \left( \frac{(1+p)^{n-1}}{r^{n-1}} \right)^{N-1}.
\]
The contribution to the local polynomial for this type of events is at most
\begin{align*}
	w_{\ec}^2(z)&=
	\sum_{3 \leq N < 2\log(n)}
		(N D^{N-2} n^8 )
		r 2^{N-2} \left( \frac{(1+p)^{n-1}}{r^{n-1}} \right)^{N-1}
		z^{nN}\\
		&=
		\frac{2 \log n\; n^{12 \log(n)} (1+p)^{n-1} z^{2n}}{r^{n-2}}
		\sum_{3 \leq N < 2\log(n)}
		\left( \frac{2 D (1+p)^{n-1}z^n}{n^6r^{n-1}}\right)^{N-2}
\end{align*}
Once again, for the chosen values of parameters, the sum is bounded by one, the function $	\frac{2 \log n\; n^{12 \log(n)} (1+p)^{n-1} z_0^{2n}}{r^{n-2}}$ is exponentially small in $n$.

\bigskip
Let us finish the proof. For large enough $n$, the value $w_{\ec}^0(z_0)+w_{\ec}^1(z_0)+w_{\ec}^2(z_0)$ is smaller than $3/n^2$.
Together with the bounds from the proof of Theorem \ref{thm:simp} this implies that all events of types $\de, \ct, \dct, \ec$ can be simultaneously avoided.
Consequently, for all large enough $n$, hypergraph $H_{(n,M)}$ is $r$-colorable provided $M< (2e)^{-4}r^{n-1}$.
Thus, there exists $\beta>0$ such that, for all $n\geq 3$, $r\geq 2$, we have the lower bound $W(n,r)>\beta r^{n-1}$.
Theorem \ref{thm:VdW} is proved.

\section{Corollaries}\label{sec:corollaries}
In this section we deduce some corollaries from our main results. We start with estimating the number of edges in uniform simple hypergraphs.

\subsection{Number of edges in simple hypergraphs with high chromatic number}

In the above sections we were focused on the maximum edge degree in uniform simple  hypergraphs with high chromatic number. Our main result, Theorem \ref{thm:simp}, states that an $n$-uniform simple non-$r$-colorable hypergraph $H$ satisfies
$$
  \Delta(H)\geqslant \alpha\cdot n\,r^{n-1}.
$$
An immediate corollary of this inequality gives a lower bound for the maximum vertex degree.
\begin{corollary}
\label{cor:maxdeg}
If $H$ is an $n$-uniform simple non-$r$-colorable hypergraph then its maximum vertex degree is at least $\alpha\cdot r^{n-1}$ where $\alpha>0$ is an absolute constant.
\end{corollary}

\bigskip
Corollary \ref{cor:maxdeg} can be used to derive a lower bound for the number of edges in a simple hypergraph with high chromatic number. This problem was raised by Erd\H{o}s and Lov\'asz in \cite{ErdLov}, they proposed to consider the extremal value $m^*(n,r)$ which is equal to the minimum possible number of edges in a simple $n$-uniform non-$r$-colorable hypergraph. Erd\H{o}s and Lov\'asz themselves proved the following bounds:
\begin{equation}\label{m*(n,r):erdlov}
  \frac{r^{2n-4}}{32n^3}\leqslant m^*(n,r)\leqslant 1600 n^4r^{2n+2}.
\end{equation}
The estimates \eqref{m*(n,r):erdlov} were improved for different relations between the parameters $n$ and $r$ in a lot of papers (see \cite{KostKumb}, \cite{Shab}, \cite{Kozik}, \cite{KR}, \cite{Kost_etal}, \cite{BFM}). The detailed history of improvements can be found, e.g., in \cite{RaigShab}. We give known bounds for $m^*(n,r)$ for fixed $r$ and large $n$.

In \cite{KostKumb} Kostochka and Kumbhat showed that
\begin{equation}\label{m*(n,r):kostkumb}
  m^*(n,r)\geqslant r^{2n-4}n^{-\varepsilon(n)},
\end{equation}
where $\varepsilon(n)>0$ slowly tends to zero for fixed $r$ and $n$ tending to infinity. Better bounds for the infinitesimal function $\varepsilon(n)$ were obtained by Shabanov in \cite{Shab} and by Kozik in \cite{Kozik}. In this paper we refine the bound \eqref{m*(n,r):kostkumb} as follows.

\begin{corollary}
\label{cor:m*(n,r)}
For any $n\geqslant 3$, $r\geqslant 2$,
$$
  m^*(n,r)\geqslant c\cdot r^{2n-4},
$$
where $c>0$ is an absolute constant.
\end{corollary}
\begin{proof}The proof is based on a trimming argument which was proposed by Erd\H{o}s and Lov\'asz and developed by Kostochka and Kumbhat. Suppose $H=(V,E)$ is an $n$-uniform simple non-$r$-colorable hypergraph. For every edge $e\in E$, fix a vertex $v(e)$ which has the maximum degree among all the vertices of $e$ (if there a few such vertices then choose one arbitrarily). Consider the following hypergraph $H'=(V,E')$ where
$$
  E'=\{e\setminus \{f(e)\}:\;e\in E\}.
$$
In other words we remove a vertex with maximum degree from every edge (a trimming procedure). Hypergraph $H'$ is $(n-1)$-uniform simple and also non-$r$-colorable. By Corollary \ref{cor:maxdeg} $H'$ contains a vertex $w$ with degree $m\geqslant \alpha\cdot r^{n-2}$. Let $f_1,\ldots,f_m$ denote the edges of $H'$ containing $w$. Consider the restored edges $e_1,\ldots,e_m$ of $H$:
$$
  e_i=f_i\cup v(e_i).
$$
Since the vertex $v(e_i)$ has the maximum degree in $e_i$, its degree in $H$ is at least $m$, the degree of $w$ in $H'$. Moreover, every two vertices of $H$ do not share more than one common edge. Thus, we obtain the following lower bound for the number of edges:
$$
  |E|\geqslant \sum_{i=1}^m(\deg v(e_i)-(i-1))\geqslant \sum_{i=1}^m(m-(i-1))\geqslant \frac{m^2}2\geqslant c\cdot r^{2n-4}.
$$
\end{proof}

For fixed $r$, our new lower bound is just $n^2$ times smaller than the upper bound proved by Kostochka and R\"odl in \cite{KR}. They showed that
$$
  m^*(n,r)\leqslant c_2\,r^{2n}n^{2},
$$
where $c_2=c_2(r)>0$ does not depend on $n$.

\subsection{Choosability in simple hypergraphs}

Finally we comment on list colorings of hypergraphs. Recall that a hypergraph $H=(V,E)$ is said to be $r$-\textit{choosable} if, for every family of sets $L=\{L(v):\; v\in V\}$ ($L$ is called a \textit{list assignment}), such that $\vert L(v)\vert=r$ for all $v\in V$, there is a proper coloring from the lists (for every $v\in V$ we should use a color from $L(v)$). It is clear that $r$-choosability implies $r$-colorability of a hypergraph. Almost all the results discussed in the introduction (except \eqref{bound:cherkozik}) hold also for list colorings, i.e. under the same conditions on the edge degree one can guarantee $r$-choosability of a hypergraph.

\bigskip
The main result of the paper can be easily extended to the case of choosability. The exact formulation is the following.

\begin{theorem}
\label{thm:simp_list}
There exists a positive constant $\alpha$ such for every $r\geq 2$, and every $n\geq 3$, any simple $n$-uniform hypergraph with maximum edge degree at most $\alpha\cdot n r^{n-1}$ is $r$-choosable.
\end{theorem}

The proof almost repeats the arguments from Sections 3 and 4. The following straightforward randomized analogue of Algorithm \ref{alg:MGC} gives a random coloring from the lists. Let $H=(V,E)$ be a hypergraph and let $L=\{L(v),v\in V\}$ be an arbitrary $r$-uniform list assignment.

\begin{algorithm-hbox}[H]
\caption{coloring from the lists of hypergraph $H=(V,E)$}\label{alg:list}
	\textbf{Input:} $c$ is a coloring from the lists $L$, $\sigma:V \to (0,1]$ injective\\
 	\While{there exists a monochromatic edge whose first non--recolored vertex $v$ is free}{
 		recolor $v$ with a random color from $L(v)\setminus \{c(v)\}$
 	} 	
\Return $c$
\end{algorithm-hbox}

The probabilistic analysis of this algorithm with given random input follows the proof of Theorem \ref{thm:simp} almost without changes.

\bibliographystyle{siam}

\end{document}